\documentclass[12pt,english]{article}
\usepackage[letterpaper]{geometry}
\usepackage[utf8]{inputenc}
\usepackage{amsmath}
\usepackage{amsthm}
\usepackage{amssymb}
\usepackage{authblk}
\usepackage{hyperref}
\usepackage{tikz}
\usepackage{epstopdf}

\newcommand{\Z}{\mathbb Z}
\usepackage{float}
\usepackage{babel,blindtext}
\usepackage{graphicx}
\usepackage{setspace}

\newtheorem{theorem}{Theorem}[section]

\newtheorem{proposition}[theorem]{Proposition}
\newtheorem{definition}[theorem]{Definition}

\title{Digital images unveil geometric structures in pairs of relatively prime numbers}

\author[1]{\sc{Benjam\'in A.~Itz\'a-Ortiz}}
\author[1]{\sc{Roberto L\'opez-Hern\'andez}}
\author[2]{\sc{Pedro Miramontes}}

\affil[1]{Centro de Investigaci\'on en Matem\'aticas, Universidad Aut\'onoma del Estado de Hidalgo, Pachuca, Hidalgo, \sc{México}}

\affil[2]{Facultad de Ciencias, Universidad Nacional Aut\'onoma de M\'exico, Mexico City, \sc{M\'exico}}

\date{June 8th, 2018}

\begin{document}

\maketitle


\section*{Introduction}
The interaction of geometry with other branches of mathematics has
proven to be  fruitful and mutually beneficial. Let us mention just
four examples of this relationship: (1) The synthesis of geometry and
algebra by Descartes and Fermat giving rise to analytic geometry; (2)
The theory of dynamical systems by Henri Poincaré, which allowed the
marriage of geometry and differential equations; (3) The use of
cutting sequences by Hedlund and Morse to exemplify the powerful
machinery of the symbolic dynamical systems theory that they created; and (4) The union between geometry and number theory accomplished by Hermann Minkowski. In this paper, we explore an example of the amazing relationship between number theory and geometry. 

In 1887, Hermann Minkowski submitted the paper {\it Räumliche Anschauung und Minima positiv definiter quadratischer}\footnote{Spatial insight and minima of positive definite quadratic forms} as a requisite to apply for a vacant position at the University of Bonn. The great J.~Dieudonné considered this report as one that  ``...contains the first example of the method which Minkowski would develop some years later in his famous {\emph{geometry of numbers}}" \cite{dieu}.

 The Geometry of Numbers \cite{minko} is Minkowski's posthumously published book pioneering the study of number theory problems in the realm of geometry. In it, Minkowski conceived geometry  as a scenario to represent complicated problems in number theory and a platform to obtain clues about its solutions by means of the {\it Räumliche Anschauung}. A fundamental tool for this purpose is the set of points in the plane with integer coordinates, the so-called standard lattice \cite{nature}. 

One way that the standard lattice has been useful is in working out
problems by simply giving an intuitive explanation of important known
mathematical concepts. For example, Figure 1 shows the standard
lattice and a  line through the origin with an irrational slope. The
fact that the line does not cross any other lattice points leads
immediately to a visualization of Dedekind's cut. This same line can be
used to define a Sturmian sequence, which is a biinfinite sequence
using the symbols 0 and 1 obtained in the following way: adding a
symbol $1$ or a symbol  $0$ every time the line crosses a horizontal
lattice line or a vertical lattice line, respectively. At the origin,
the only point where the line crosses a horizontal and vertical
lattice line simultaneously, we decide to add a word $01$ rather than choosing between a symbol 0 or 1. This construction characterizes all the biinfinite sequences in the symbols 0 and 1 such that the differences in the number of of 0's between any two consecutive symbols 1's is at most 1. 
Finally, as an illustration of the density of the  rational numbers in the reals, let us consider the particular case when the line has a slope equal to the golden mean $\Phi=\dfrac{1+\sqrt{5}}{2}$. Since it is known that the sequence  $\dfrac{F_{n+1}}{F_n}$, where $F_n$ is the $n$-th term of the familiar Fibonacci sequence, converges to $\Phi$, then the family of lines (through the origin)  having slopes equal to  $\dfrac{F_{n+1}}{F_n}$ converges to the line having slope $\Phi$. Therefore, since the pairs of relatively prime numbers  $(F_{n},F_{n+1})$ determine the family of lines, we may say that there is a sequence of pairs of relatively prime numbers that represent the golden mean number. 

In its origins in ancient Greece, geometry was an eminently visual
discipline. With the increasing level of abstraction that it attained
in the XIX century, it nearly left aside any pictorial
representation. Minkowski's book does not have a single
illustration. In this essay, we show that with the advent of digital computers, images can play again a central role in the understanding and development of geometry.

With this purpose, we  show that when  plotting  a certain large collection of relatively prime numbers, a family of  quadratic arcs emerges. The appearance of these arcs was unexpected and, to the best of our knowledge, has not been observed before. We divide this work into two sections. In the first section, we introduce the  B\'ezout transformations and use them to generate a special family of relatively prime numbers. In the second section, we state our main results which justify the appearance of quadratic arcs in the graphs introduced in Section 1 and also their symmetries.

The authors gratefully acknowledge the valuable suggestions from the referees which helped to improve this paper.

\begin{figure}[ht]
\centering
\includegraphics[scale=0.6]{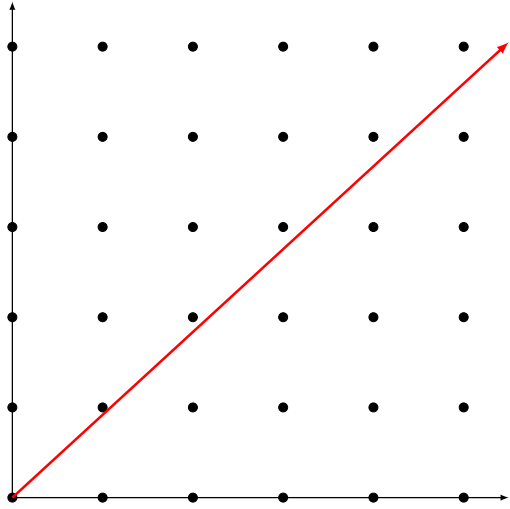}
\caption{The standard lattice and a line with irrational slope}
\label{latt}
\end{figure}

\section{B\'ezout set}

In this section we introduce the set of pairs of relatively prime numbers that we study in this paper. The fundamental theorem of arithmetic states that any positive integer greater than one is either prime or the product of prime numbers and that this decomposition is unique except for the order of the factors.  Two positive integers $p$ and $q$  are said to be relatively prime if their only common (positive) divisor is 1. Notice that when $p$ is a
prime number, then every positive number $q$ less that $p$ is
relatively prime to $p$; however, if $p$ is not prime, then the number of positive integers relatively prime to $p$ is given by the well-known Euler's totient function $\phi(p)$. Therefore, if $p$ is prime, then $\phi(p)=p-1$.

If one draws dots in the Cartesian plane each representing a pair of relatively prime numbers, it is intriguing to discover a fractal-like picture; see, for example, \cite{wolfram}. In Figure~\ref{figRP}, we draw all pairs of relatively prime numbers $(p,q)$ with $1\leq q \leq p\leq 1000$.

\begin{figure}[ht]
\centering
\includegraphics[scale=0.4]{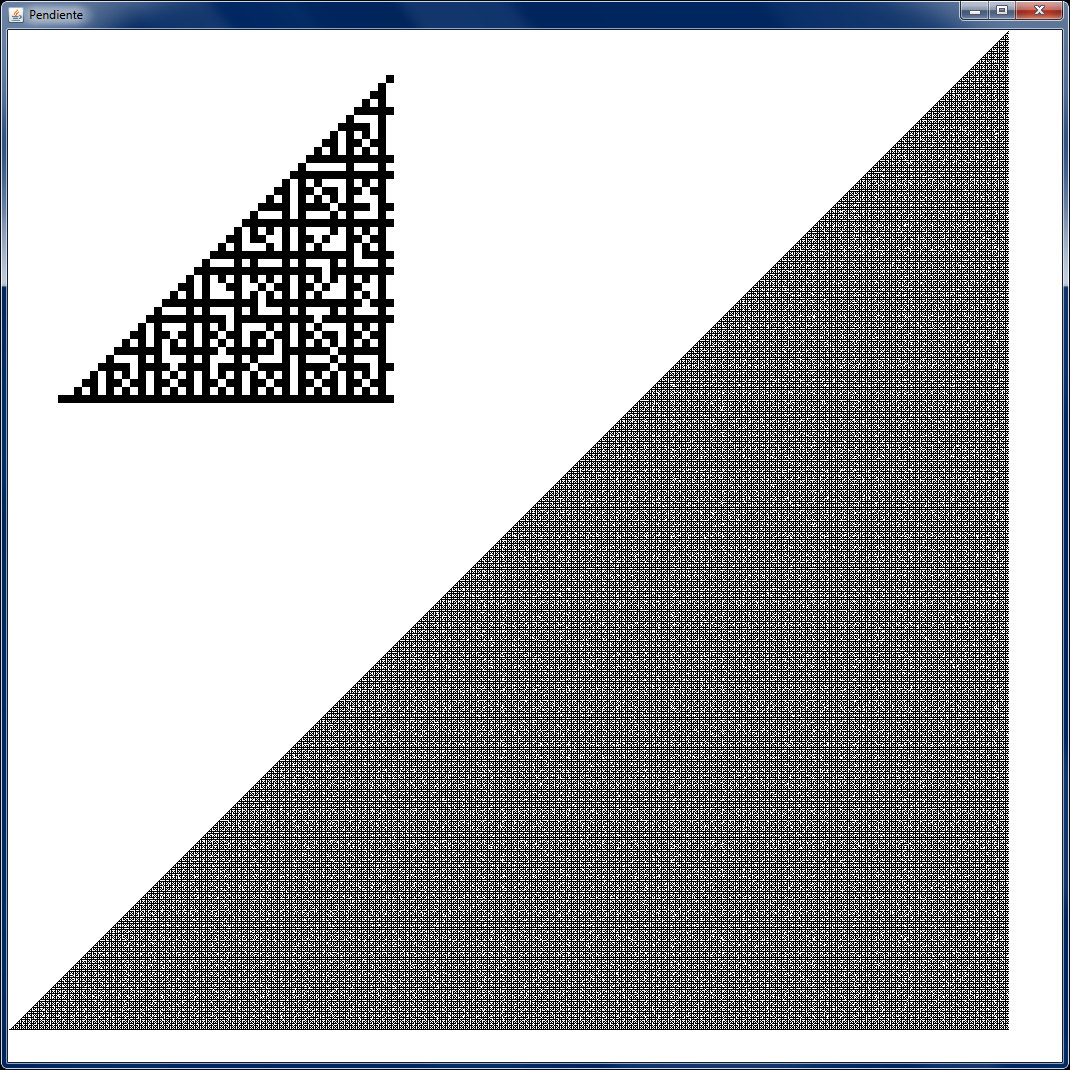}
\caption{Pairs $(p,q)$ of relatively prime numbers with $0<q\leq p\leq 1000$. A detail of the graph is shown in the box on the top left-hand side.}
\label{figRP}
\end{figure}

As we mentioned at the beginning of this section, our goal is to introduce a way to select a subset of pair of relatively prime numbers. First let us introduce the following notation.
\begin{definition}
Given an integer $p>1$ and an integer $1<q$ relatively prime to $p$, for each $i=-1,0,1$, we define  the {\em $i$-B\'ezout coefficients} of $(p,q)$ as the unique pair of relatively prime numbers $B_i(p,q)=(a_i,b_i)$ such that $0<a_i \leq p$, $0< b_i \leq q$  and  satisfy the {B\'ezout identity}
\begin{equation}\label{BezoutID}
    a_iq-b_ip=i.
\end{equation}
 \end{definition}
 We refer to  \cite[Theorem~5.1]{niven} 
 for a proof of the existence of $B_{\pm 1}(p,q)$. The case $i=0$ is actually trivial, as one may choose $a_0=p$ and $b_0=q$; however, the definition of $B_0$ may well be extended to all pairs of positive integers $(p,q)$ by letting $B_0(p,q)$ to be the unique pair of positive relatively prime integers $(a_0,b_0)$ such that $p/q=a_0/b_0$ 
 
  For example, $B_1(6,5)=(5,4)$, $B_{-1}(6,5)=(1,1)$, $B_0(6,5)=(5,6)$ and $B_1(5,2)=(3,1)$, $B_{-1}(5,2)=(2,1)$, $B_0(5,2)=(2,5)$.

 It becomes apparent that by fixing a positive $p$ greater than 1, and plotting,  for each $1<q<p$ relatively prime to $p$, the pairs   $B_{\pm 1}(p,q)$ and $B_{\pm 1}(q,p)$, then the generated  graph contains some  intriguing arcs. In Figures~\ref{3.1}, \ref{3.2} and \ref{3.3} we show a sequence of such graphs with $p=10^6$.

\begin{figure}[!htb]
\minipage{0.32\textwidth}
  \includegraphics[width=\linewidth]{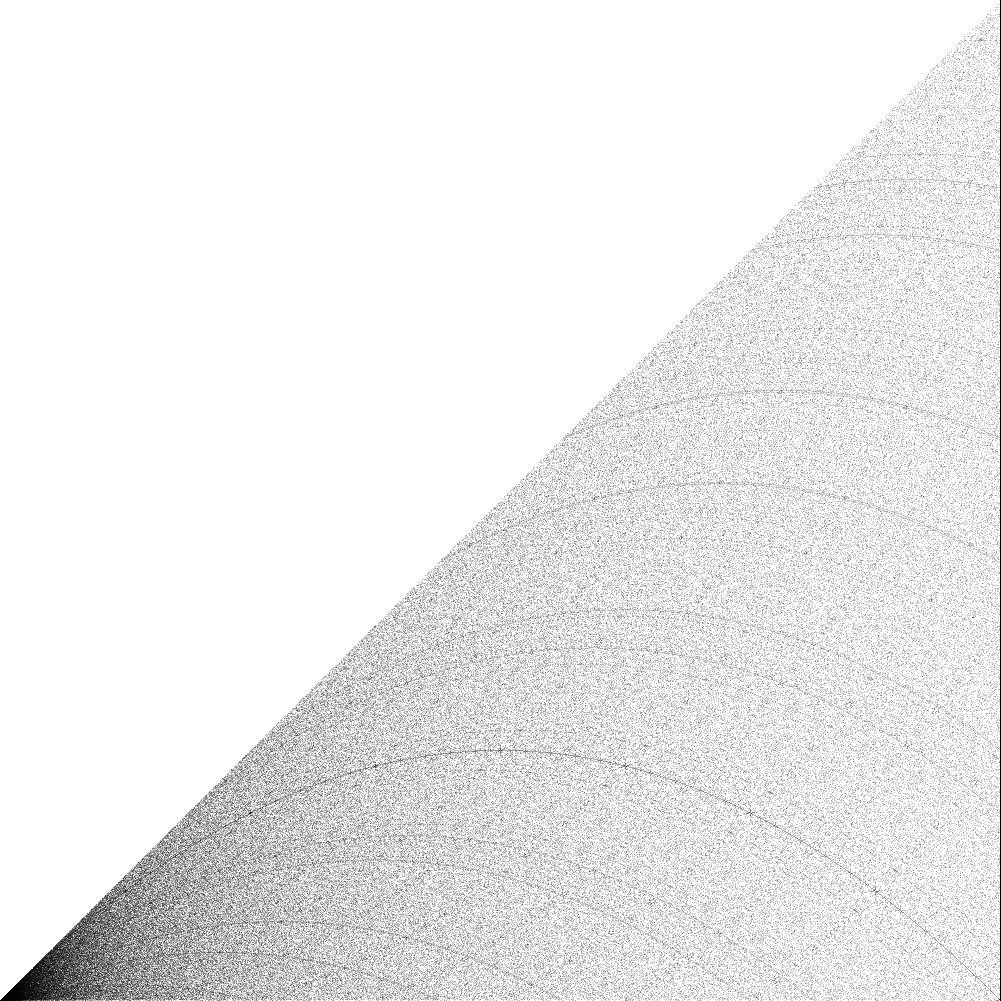}
  \caption{Plotting points $B_1(p,q)$ with  $1<q<p=10^6$ relatively prime.}
  \label{3.1}
\endminipage\hfill
\minipage{0.32\textwidth}
  \includegraphics[width=\linewidth]{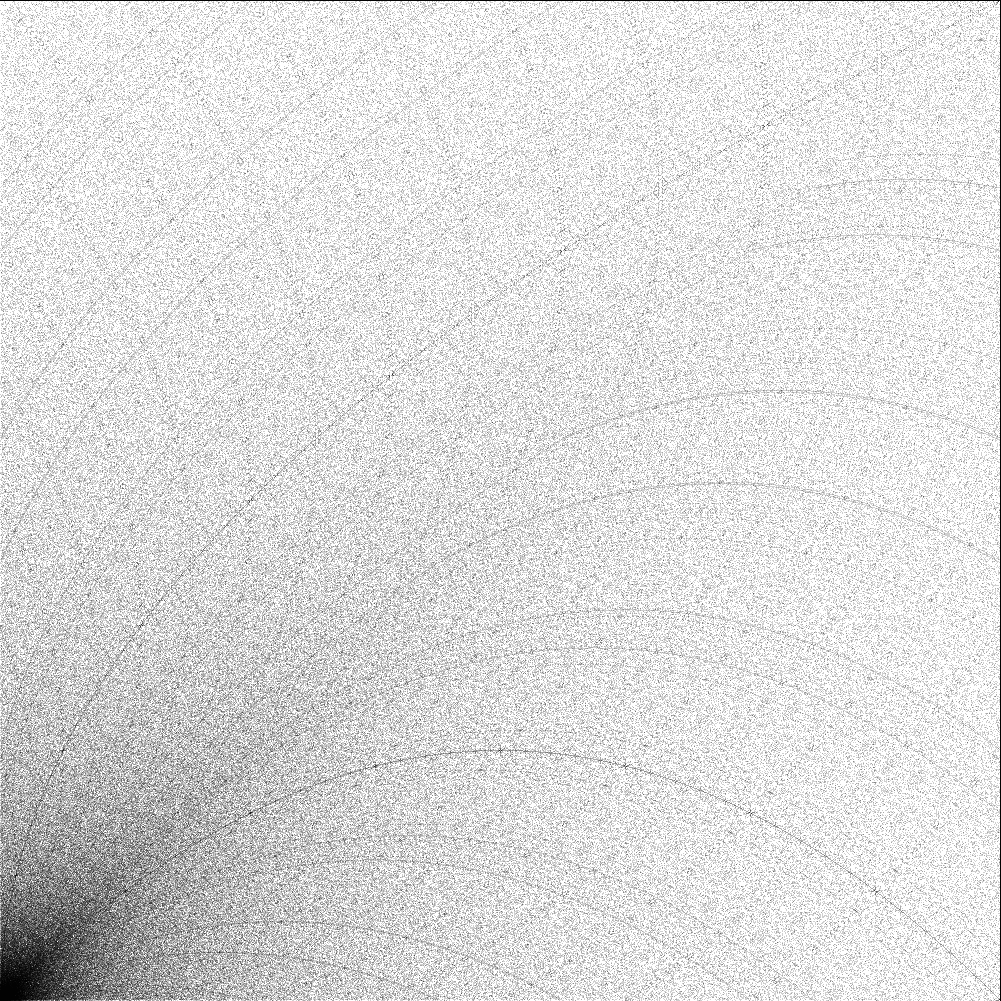}
  \caption{Plotting  $B_1(p,q)$ and $B_1(q,p)$  with  $1<q<p=10^6$ relatively prime.}\label{3.2}
\endminipage\hfill
\minipage{0.32\textwidth}%
  \includegraphics[width=\linewidth]{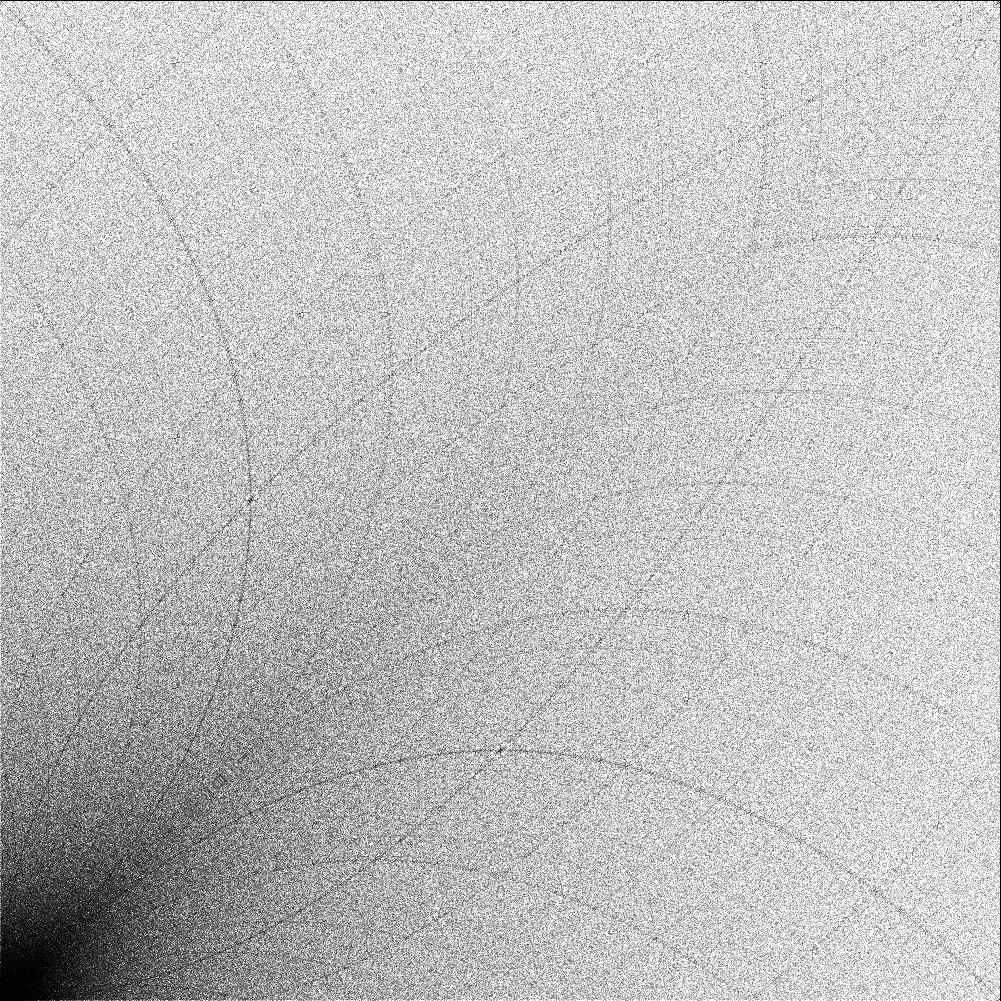}
  \caption{Plotting $B_{\pm 1}(p,q)$ and $B_{\pm 1}(q,p)$  with  $1<q<p=10^6$ relatively prime.}\label{3.3}
\endminipage
\label{3figures}
\end{figure}

 Finally, we notice that, defining $B_1(p,q)=\left( a|q|/q,b|p|/p\right)$, where $B_1(|p|,|q|)=(a,b)$, one extends the B\'ezout transformation $B_1$ to all pairs of relatively prime integer numbers with absolute value greater than one. 
Additionally, $B_{-1}$ may be extended to
all pairs of relatively prime integer numbers $p$ and $q$ with absolute values greater than one, by the
formula $B_{-1}(p,q)=\left(a\sp\prime |q|/q,b\sp\prime |p|/p\right)$
where $B_{-1}(|p|,|q|)=\left(a\sp\prime,b\sp\prime\right)$. 
Analogously,  $B_0$ may be defined for all pairs $(a,b)\not=(0,0)$ of integers by defining $B_{0}\left(a,b\right)=\left(r p/|p|,s q/|q|\right)$, where $B_0(|a|,|b|)=\left(r,s\right)$.

 We have arrived at the desired set of relatively prime numbers, which we define next.

 \begin{definition}
 For a fixed integer number $p>1$ we define $\mathcal B_p$, the {\em B\'ezout set for $p$}, as the set consisting of all $i$-B\'ezout coefficients of $(\pm p,\pm q)$ and $(\pm q,\pm p)$, for $i=\pm 1$, where $1<q<p$ is relatively prime to $p$.
 \end{definition}
 
 In Figure~\ref{fig1} and Figure~\ref{fig2} we plot the B\'ezout set $\mathcal B_p$ for  $p=1,000,000$ and $p=250,000$, respectively. To discard the possibility of an optical artifact, we include  Figure~\ref{fig3}, where the graph of $\mathcal B_{512}$ is shown in $1:1$ scale. It is interesting to notice that the arcs that seem to be forming in these figures appear when $p$ has at least one repeated prime in its prime factorization. For example, for $p=317811=(3)(13)(29)(281)$, which incidentally is also a Fibonacci number, it has a B\'ezout graph with no discernible arcs, while for $p=46368=(2^5)(3^2)(7)(23)$, also a Fibonacci number, the arcs clearly show up. See Figure~\ref{fibonacci}.

\begin{figure}[H]
\centering
\includegraphics[scale=0.18]{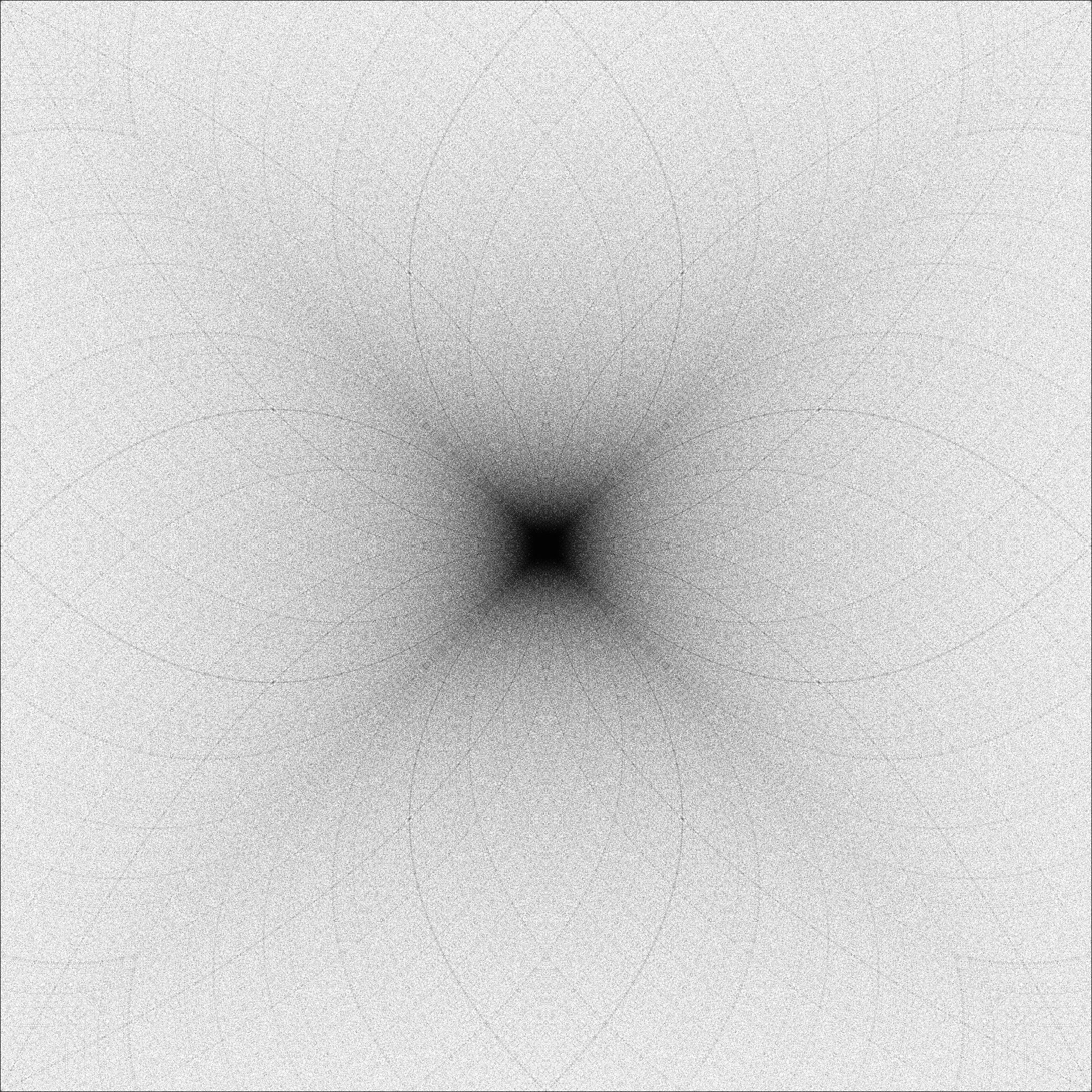}
\caption{B\'ezout set for $p=1000000$.}
\label{fig1}
\end{figure}

\begin{figure}[H]
\centering
\includegraphics[scale=0.18]{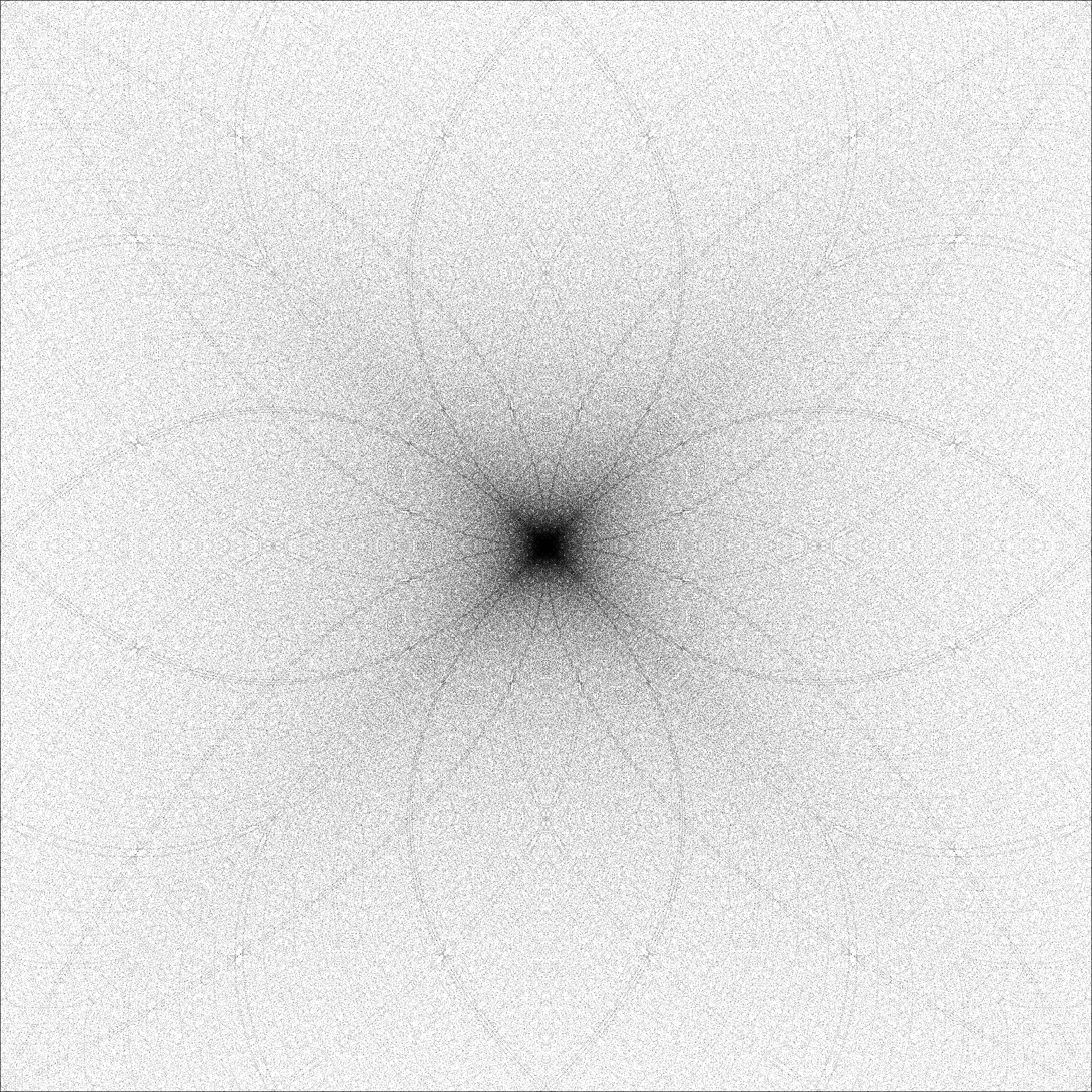}
\caption{B\'ezout set for $p=250000$.}
\label{fig2}
\end{figure}

 \begin{figure}[htb]
\centering
\includegraphics[scale=0.3]{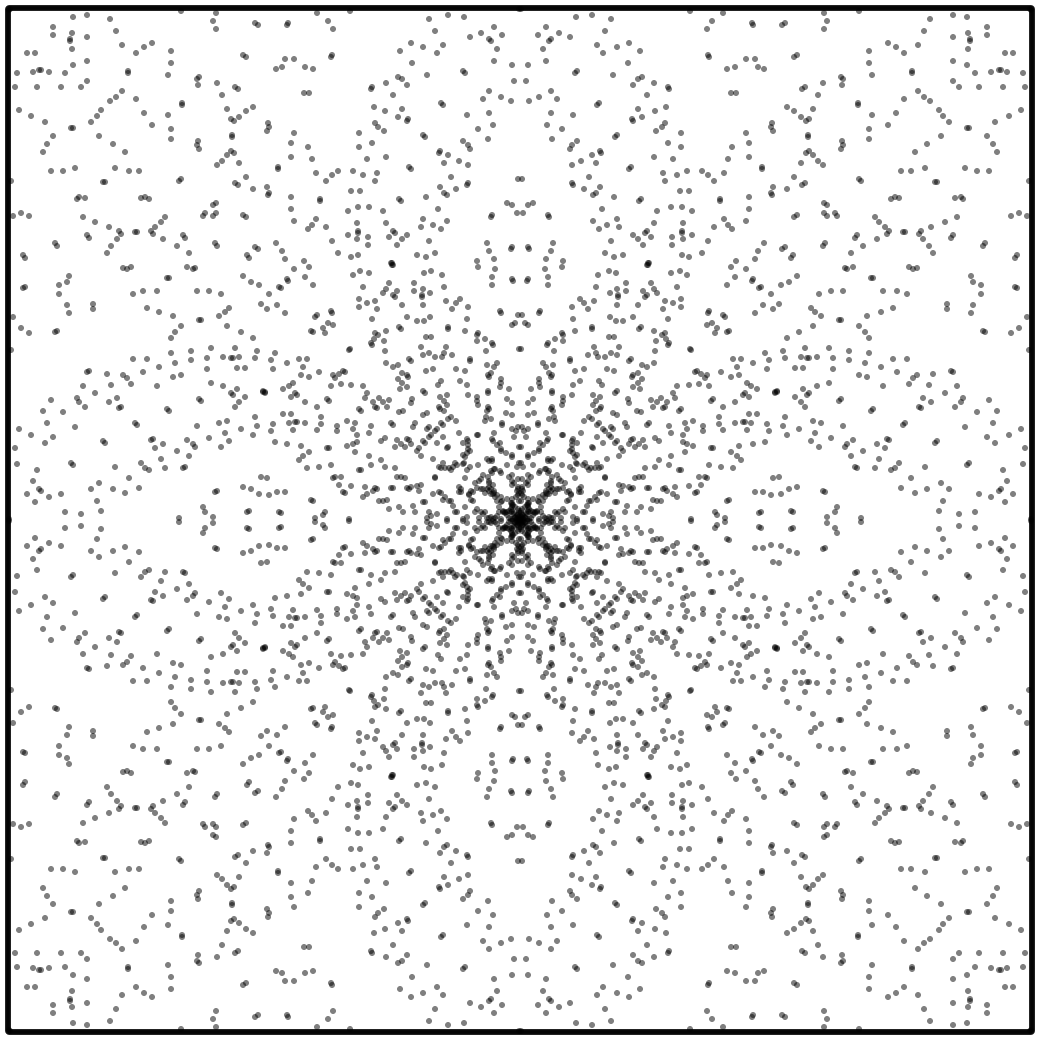}
\caption{A 1:1 scale graph of B\'ezout set for $p=512$.}
\label{fig3}
\end{figure}

To plot the B\'ezout coefficients in Figures~\ref{3.1}--\ref{fig3}, the extended
Euclidean algorithm is used. The extended Euclidean algorithm uses the
Euclidean algorithm repeatedly to produce the coefficients for the B\'ezout identity. Specifically, for $(p,q)$ relatively prime and positive, it produces a pair of integer coefficients $(x,y)$ satisfying the following:
\[
   xq + yp = 1.
\]

Note that then, one of the coefficients $x$ or $y$  must be positive and the other one negative or zero.
If $y>0$, then $x\leq 0$, and we obtain $B_{-1}(p,q)=(|x|,y)$, while if $y\leq 0$ and $x > 0$, then $B_1(p,q)=(x,|y|)$.
If we obtained a pair $(a,b)=B_{-1}(p,q)$, we calculate a new pair applying the transformation $a' = p - a$
and  $b' = q - b$ so that $(a', b')=B_1(p,q)$.
Similarly, if $(a, b)=B_1(p,q)$ then $B_{-1}(p,q)=(p-a,q-b)$.

We conclude this section relating the B\'ezout set $\mathcal B_p$ with the Farey sequence $F_p$, as a motivation for
 restricting our attention to  B\'ezout coefficients. Recall that a Farey sequence $F_p$ of order $p$ is the set of
 all reduced nonnegative fractions less than or equal to 1 and with denominators less  than or equal to $p$, arranged in increasing order. For example, 
 \[
 F_5=\left\{\frac{0}{1},\frac{1}{5},\frac{1}{4},\frac{1}{3},\frac{2}{5},\frac{1}{2},\frac{3}{5},\frac{2}{3},\frac{3}{4},\frac{4}{5}, \frac{1}{1} \right\}
 \]
 and 
\[
F_6=\left\{\frac{0}{1},\frac{1}{6},\frac{1}{5},\frac{1}{4},\frac{1}{3},\frac{2}{5},\frac{1}{2},\frac{3}{5},\frac{2}{3},\frac{3}{4},\frac{4}{5},\frac{5}{6}, \frac{1}{1}\right\}
\]

For a fixed positive integer $p$, notice that there is a one-to-one correspondence between the elements $b/a$ in $F_p$ and the points $(a,b)$, which belong to the set of all pairs of relatively prime numbers with $0<b\leq a\leq p.$ For example, if we graph all points $(a,b)$ such that $b/a$ belongs to $F_p$ with $p=1000$, then we recover Figure~\ref{figRP}.

Farey sequences have remarkable properties. One such property is the following. Given $q/p\in F_p$, the predecessor $b/a$ of $q/p$ in $F_p$
satisfies B\'ezout's relation $aq-bp=1$ \cite[Theorem~5.3, pp 98]{apostol}. Hence, if $b/a$
is the predecessor of $q/p$ in the Farey sequence $F_p$ then $B_1(p,q)=(a,b)$. 
Hence, in other words, we are proposing  to highlight  neighbors of the  elements of the form $q/p$ in the Farey sequence $F_p$. Now, we already pointed out that the B\'ezout coefficients $B_1(p,q)=(a,b)$ of $(p,q)$ correspond to the predecessor $b/a$ of $q/p$ in $F_p$. Analogously, we observed  that the  successor $d/c$ of $q/p$ in $F_p$ satisfies the equality
\begin{equation}\label{B-1}
     cq-dp=-1
\end{equation}
\noindent
and $0< c \leq p$ and $0< d \leq q$. We remark that the $q/p$ in $F_p$ is the mediant sum of its neighbors $b/a$ and $d/c$, that is, $q/p=(b+d)/(a+c)$, see \cite[Theorem~5.2, pp 98]{apostol}.

 \begin{figure}[!htb]
\minipage{0.49\textwidth}
  \includegraphics[width=\linewidth]{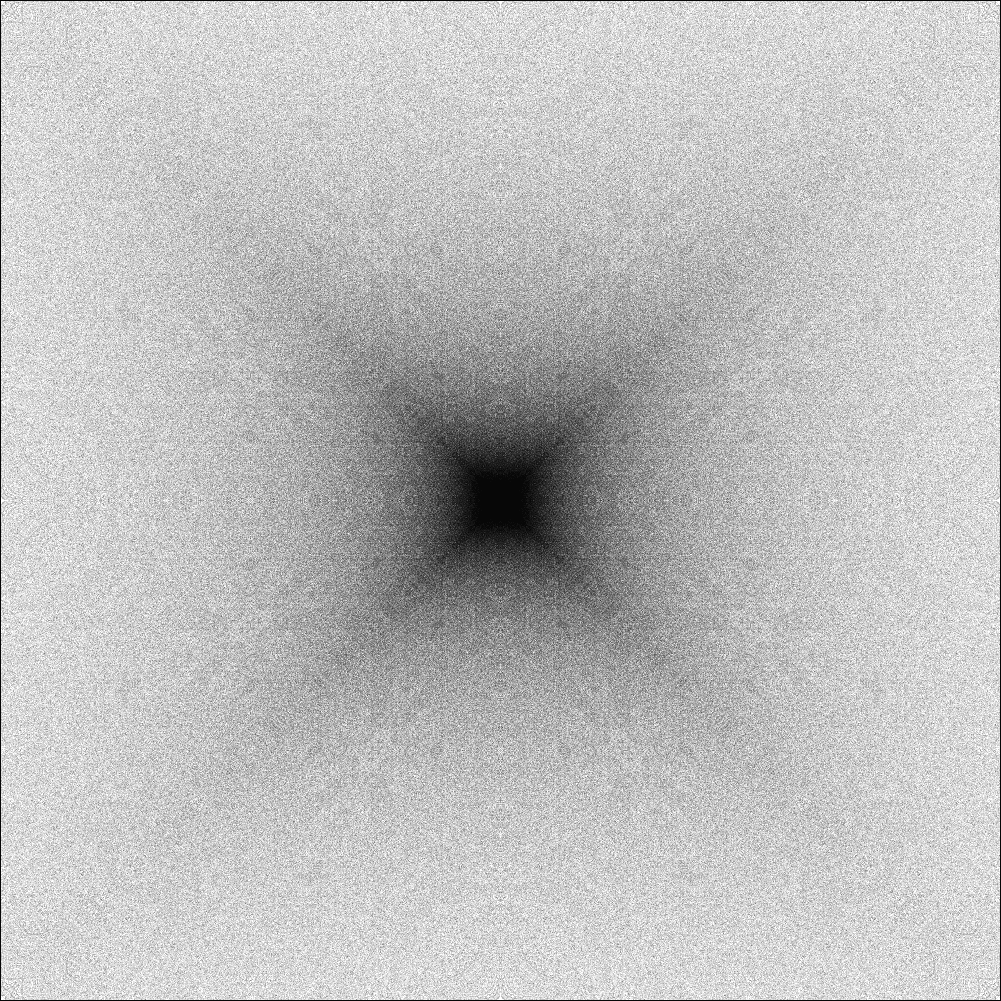}
\endminipage\hfill
\minipage{0.49\textwidth}
  \includegraphics[width=\linewidth]{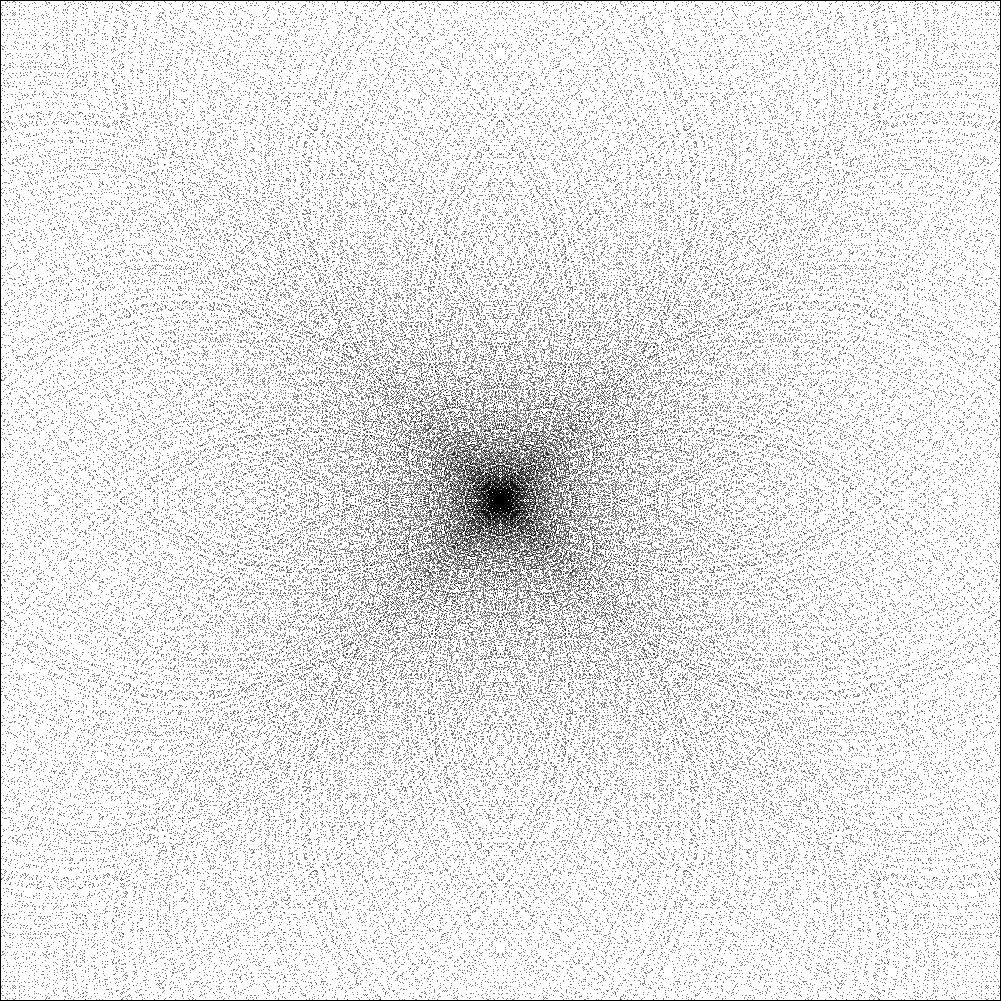}
\endminipage
\caption{B\'ezout sets for $p=317811$ (left) and $p=46368$ (right)}
\label{fibonacci}
\end{figure}

\section{The Geometry of numbers}

In this section we state and prove our main results. They are meant to justify  the appearance of the quadratic arcs and the symmetries in the graphs of the B\'ezout sets.

\subsection{Quadratic arc}

Let us begin by fixing an integer $p>1$  and suppose that $1<q<p$ is an  integer relatively prime to $p$. Suppose  $B_1(p,q)=(a,b)$. We next show
how to derive an arithmetic sequence of relatively prime pairs
$\left\{\left(p,q+nd\right)\right\}_{n\in\mathbb Z}$  so that their B\'ezout
coefficients satisfy a certain quadratic relation. 
 We obtain the arithmetic sequences as follows. Let $w$ be the smallest positive integer such that $pw$ is a perfect square. Let $B_0\left(\sqrt{pw},q-a\right)=(r,s)$. Put $d=r\sqrt{pw}$.
Hence, $\frac{d^2}{p}=r^2w$ is an integer. On the other hand, since $$\frac{sw}{d}=\frac{sw}{r\sqrt{pw}}=\frac{q-a}{\sqrt{pw}}\frac{w}{\sqrt{pw}}=\frac{q-a}{p}$$ we get  $\frac{q-a}{p}d=sw$ is also an integer and thus $b - \tfrac{q-a}{p} \, d  - \tfrac{1}{p}\,  d^2 $ is an integer. Hence, for $n\in\Z$ we get
\begin{align*}
  \left( a-nd\right) (q+nd)- \left( b - \frac{q-a}{p} \, (nd)  - \frac{1}{p}\,  (nd)^2 \right )p &= aq-bp\\
                                  &= 1.
\end{align*}
which proves that $p$ and $q+nd$ are relatively prime. Moreover, whenever $1<q+nd<p$, this shows that 
$B_1(p,q+nd)=\left( a-nd,\ b - \frac{q-a}{p} \, (nd)  - \frac{1}{p}\,  (nd)^2  \right)$. We claim that $B_1(p,q+nd)$ lie in a quadratic curve. Indeed, by making $x=a-nd$ and $y= b - \frac{q-a}{p} \, (nd)  - \frac{1}{p}\,  (nd)^2 $ we obtain 
\begin{align*}
     y &= b - \frac{q-a}{p} \, (nd)  - \frac{1}{p}\,  (nd)^2 \\
       &= b - \frac{q-a}{p} \, (a-x)  - \frac{1}{p}\,  (a-x)^2\\
       &= b-\frac{q}{p}a + \left(\frac{q-a}{p}+\frac{2a}{p}\right)x- \frac{1}{p}x^2\\
       &= -\frac{1}{p}+\frac{a+q}{p}x- \frac{1}{p}x^2\\
       &= a_0 + a_1 x + a_2 x^2
\end{align*}
where  $a_1=\tfrac{a+q}{p}$ and $a_0=a_2=-\frac{1}{p}$ are rational numbers.

\noindent We have then proved  the following.
\begin{theorem}\label{formula}
   Let $p>1$ be an integer and let $1<q<p$ be relatively prime to $p$. Then there is an arithmetic sequence $\left\{q+nd\right\}_{n\in\mathbb Z}$  of integers relatively prime to $p$ such that the $1$-B\'ezout coefficients of the pairs $\left\{\left(p,q+nd\right)\right\}_{n\in\mathbb Z}$  lie in a quadratic curve, that is, there are real numbers $a_0,a_1,a_2$ so that for each $n\in\Z$, the pair $B_1(p,q+nd)$ satisfy the equation $y=a _0+a_1x+a_2x^2$. 
\end{theorem}

For example, for $p=1024$, let $q=817$. We compute $B_1(p,q)=(465,371)$. In this case, $w=1$, since $p$ is a perfect square. Then, $B_0(\sqrt{pw},q-a)=B_0(32,352)=(1,11)=(r,s)$. Hence, the common difference is $d=32$. We list some of the obtained points in Table~\ref{1arco} and illustrate them in Figure~\ref{fig4}. 

\begin{table}[hbt]
\centering
\begin{tabular}{|l|l|c|} \hline
$n$   & $(p,q+nd)$  &  $B_1(p,q+nd)$ \\ \hline
$-2$  & $(1024,753)$  & $(529,389)$ \\
$-1$  & $(1024,785)$  & $(497,381)$  \\
$0$   & $(1024,817)$  & $(465,371)$  \\
$1$   & $(1024,849)$  & $(433,359)$  \\
$2$   & $(1024,881)$  & $(401,345)$  \\
$3$   & $(1024,913)$  & $(369,329)$ \\ \hline
\end{tabular}
\caption{The points of the form $B_1(1024,817+32n)$ in the right column belong to a quadratic arc.}
\label{1arco}
\end{table}

\begin{figure}[htb]
\centering
\includegraphics[scale=0.4]{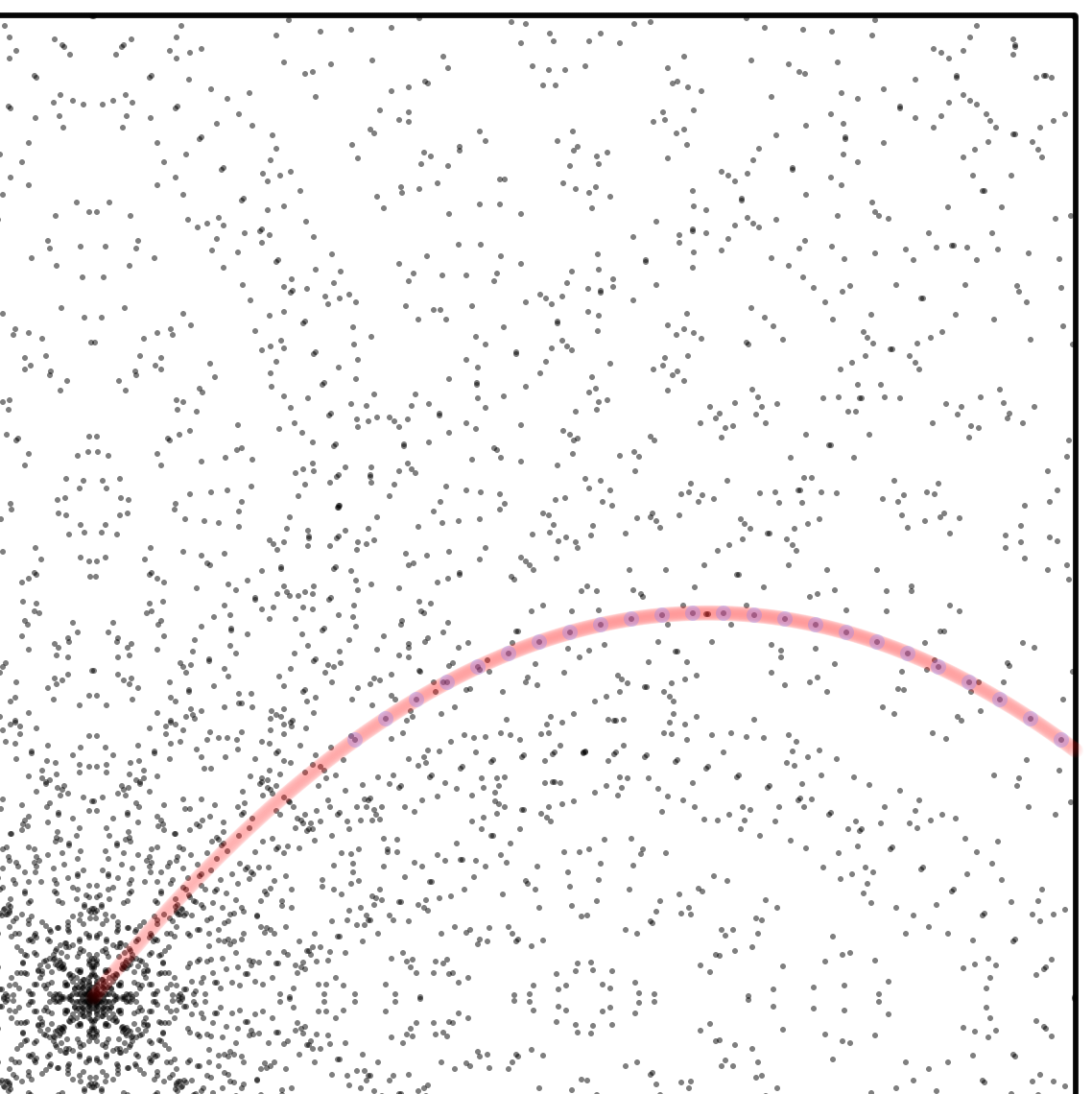}
\caption{Identifying B\'ezout coefficients in a quadratic arc. Here, $p=1024$ and some of the highlighted points are listed in Table~\ref{1arco}.}
\label{fig4}
\end{figure}

In fact, when fixing $p$ and repeating
this exercise for all $q$ relatively prime to $p$, the graph obtained by plotting the corresponding B\'ezout coefficients of the pairs in the sequence turns out to   recover some of the arcs in Figures~\ref{fig1}--\ref{fig3}. We point out that while this argument exhibits quite clearly why the points in Figure~\ref{fig1} are points of a quadratic function, the common difference $d$ in the arithmetic sequence $q+nd$ is sometimes too large to lie in the B\'ezout set since we are restricted to integers $1<q<p$.

\subsection{Symmetries}

Let $F$ denote the flip transformation, that is, $F(p,q)=(q,p)$ for every pair of integers $(p,q)$. The following proposition justify some of the symmetries in the B\'ezout set. It is straightforward so we state it without proof.

\begin{proposition} The following properties hold.
   \begin{enumerate}
      \item[(a)] $F$ is invertible and $F^{-1}=F$.
      \item[(b)] $B_{-1}\circ F=F\circ B_1$. Here, the symbol $\circ$ denotes composition of functions. 
      \item[(c)] $B_{-1}+B_1=B_0$. Here, the addition is the usual component-wise addition of pairs of numbers, that is, $(a_1,b_1)+(a_2,b_2)=(a_1+a_2,b_1+b_2)$. This operation defines a group structure on $\Z^2=\Z\times\Z$. This property corresponds to the mediant sum of Farey sequences.
    \end{enumerate}
\noindent Furthermore, if $B_1(p,q)=(a,b)$, with $p,q>1$ relatively prime, then it follows that
    \begin{enumerate}
      \item[(d)] $B_{-1}(p,p-q)=(a,a-b)$
      \item[(e)] $B_1(p,a)=(q,b)$
      \item[(f)] $B_1(q,p)=(q-b,p-a)$
  \end{enumerate}
  \end{proposition}

 In Theorem~\ref{formula} we were able to  provide a formula  that justifies the appearance of quadratic arcs in the B\'ezout sets. This formula, however, is not exhaustive, that is, it is possible to identify quadratic arcs in some B\'ezout sets were the common difference in the arithmetic sequence involved is too large so that it misses or skips some points over the quadratic arcs. For this reason, we now give a second argument to explain the formation of the remarkable symmetries presented in Figures~\ref{fig1}, \ref{fig2} and \ref{fig3}. This argument, unlike the given before, is algebraic in nature.
Let $p>1$ be an integer. Recall the notation $a\equiv b \mod p$
 means that $p$ divides the difference $a-b$. 
 
 \begin{definition}
 We define the set $\Z_{p}^{*}$ as the set of all positive integers less than $p$ and relatively prime to $p$. It is an Abelian group of order $\phi(p)$, with the operation multiplication module $p$, which we shall denote by $\bullet$
 \end{definition}
 
 The reader is referred to \cite[Theorem~2.4.7, pp.~62]{herstein} for more details about this group. Notice that the group unit in $\Z_{p}^{*}$ is the number $1$. As an example, suppose $p=9$; then, we have $\Z_{9}^{*}=\{1,2,4,5,7,8\}$, and the product $2\bullet 5=1$ in $\Z_{9}^{*}$ since $(2)(5)=10\equiv 1 \mod{9}$. Similarly, $4\bullet 7=1$ in $\Z_{9}^{*}$, this also says that $4^{-1}=7$ in $Z_{9}^{*}$.


Let $G$ be an Abelian group. It is straightforward to check that the mapping $\theta\colon G\to G$ given by $\theta (g)=g^{-1}$ is an automorphism.
We obtain the following theorem.

\begin{theorem}
 Let $p>1$ be an integer and let $\Z_p^\ast$ be the Abelian group of positive integers relatively prime to $p$ and less than $p$ with multiplication module $p$. Consider the automorphism $\theta_p\colon Z_p^\ast\to Z_p^\ast$ defined as $\theta_p(q)=q^{-1}$. Then, for any $q\in Z_p^\ast$ we have 
 \begin{equation*}
     \theta_p(q)\equiv q^{\phi(p)-1}\mod p
 \end{equation*}
 where $\phi$ is Euler's totient function. Moreover
 \begin{equation*}
    B_1(p,q)=\biggl(\theta_p(q),\dfrac{q\ \theta_p(q)-1}{p}\biggr).
 \end{equation*}
\end{theorem}

\begin{proof}
Let $q\in\Z_p^\ast$. Since Euler's theorem \cite[Theorem~2.4.8, pp.~63]{herstein}  states that $q^{\phi(p)}\equiv 1\mod p$, then we have that $q\theta_p(q)\equiv q\, q^{\phi(p)-1}\mod p$. But also $q\,\theta_p(q)\equiv 1 \mod p$. Thus $\theta_p(q)\equiv q^{\phi(p)-1}\mod p $, as wanted.

Finally, since $q\bullet\theta_p(q)=1$ in $\Z_p^\ast$, we have that  $\dfrac{q\ \theta_p(q)-1}{p}$ is an integer. We then compute
\begin{align*}
     q\,\theta_p(q)-p\dfrac{q\ \theta_p(q)-1}{p}
     &= q\,\theta_p(q) - q \left( \theta_p(q)-1\right)\\
     &= 1
\end{align*}
and obtain $B_1(p,q)=\biggl(\theta_p(q),\dfrac{q\, \theta_p(q)-1}{p}\biggr)$, as was to be proved.
\end{proof}

 For example, for $4\in\Z_9^\ast$, we computed before $\theta_9(4)=7$ and so $\frac{q\,\theta_p(q)-1}{p}=\frac{4(7)-1}{9}=3$. Thus $B_1(9,4)=(7,3)$, which gives us an alternative way to compute B\'ezout coefficients.

Incidentally, the integer $\frac{q^{\phi(p)}-1}{p}$, obtained from Euler's theorem, is known as Euler's quotient. Since an automorphism of a finite group can be regarded as a rule that maps the elements of the group in such a way that geometry, in terms of algebra,  is preserved, 
hence, it can be argued that the symmetries exhibited in Figures~\ref{fig1}, \ref{fig2} and \ref{fig3} are somehow the symmetries preserved by the group automorphism $\theta_p$. Indeed, when plotting the B\'ezout set $\mathcal B_p$, we essentially are plotting points of the form  $B_1(p,q)=\biggl(\theta_p(q),\dfrac{q\ \theta_p(q)-1}{p}
\biggr)$, with  $q\in \Z_{p}^{*}$. Thus, the first component of $B_1(p,q)$ is nothing but the image of an automorphism, while the second component may be regarded as an Euler quotient.

\end{document}